\documentclass[a4paper,10pt,leqno]{amsart}
\date{June 25, 2019}
\usepackage{graphicx} 
\usepackage{verbatim,enumerate}
\usepackage{color}
\usepackage[dvipsnames]{xcolor}
\title[]{%
Improvement
of the Bernstein-type theorem for 
space-like zero mean curvature
graphs in Lorentz-Minkowski space
using fluid mechanical duality
}

\author[]{
S.~Akamine, M.~Umehara and   % Masaaki Umehara
        K.~Yamada % Kotaro Yamada
}   

\address[Shintaro Akamine]{%
Graduate School of Mathematics, 
Nagoya University, Chikusa-ku, Nagoya 464-8602, Japan
}
\email{s-akamine@math.nagoya-u.ac.jp}

\address[Masaaki Umehara]{%
   Department of Mathematical and Computing Sciences,
   Tokyo Institute of Technology,
   Tokyo 152-8552, Japan
}
\email{umehara@is.titech.ac.jp}

\address[Kotaro Yamada]{%
   Department of Mathematics,
   Tokyo Institute of Technology,
   Tokyo 152-8551, Japan
}
\email{kotaro@math.titech.ac.jp}

\thanks{
Umehara was partially supported by 
the Grant-in-Aid for Scientific Research  (A) No.\ 26247005,  
and Yamada by (B) No.\  17H02839 from Japan Society for the 
Promotion of Science.
}

\usepackage{amsthm}
\theoremstyle{plain}

 \newtheorem{theorem}{Theorem}[section]

 \newtheorem{fact}[theorem]{Fact}
 \newtheorem{lemma}[theorem]{Lemma}

\theoremstyle{definition}
 \newtheorem{definition}[theorem]{Definition}
\theoremstyle{remark}
 \newtheorem{remark}[theorem]{Remark}
 \newtheorem*{remark*}{Remark}
\newtheorem{example}[theorem]{Example}
 \newtheorem*{acknowledgement}{Acknowledgement}

\numberwithin{equation}{section}

\newcommand{\op}[1]{{\operatorname{#1}}}

\newcommand{\mb}[1]{\vect{#1}}

\newcommand{\vect}[1]{\boldsymbol{#1}}
\newcommand{\R}{\boldsymbol{R}}

\renewcommand{\phi}{\varphi}

\newcommand{\pmt}[1]{{\begin{pmatrix} #1  \end{pmatrix}}}

\begin{document}
\maketitle

\begin{abstract}
Calabi's Bernstein-type 
theorem asserts that
{\it a zero mean curvature
entire graph in Lorentz-Minkowski space
$\mb L^3$
which admits only space-like points is a space-like plane}.
Using the fluid mechanical duality between
minimal surfaces in Euclidean 3-space $\mb E^3$
and maximal surfaces in Lorentz-Minkowski space
$\mb L^3$, we give
an improvement of this Bernstein-type 
theorem.  More precisely, we show that
{\it a zero mean curvature
entire graph in $\mb L^3$
which does not admit time-like points
$($namely, a graph consists of only space-like and
light-like points$)$ is a  plane}.
\end{abstract}

%%%%%%%%%%%%%%%%%%%%%%%%%%%%%%%%%%%%%%%%%%%%%%%%
\section{Introduction} \label{sec:1} 
Consider a $2$-dimensional barotropic steady flow on 
a simply connected domain $D$ in the $xy$-plane $\R^2$
whose velocity vector field is $\vect{v}=(u,v)$, 
with  density $\rho$ and pressure $p$.
We assume there are no external forces.
Then
\begin{itemize}
\item the flow is a foliation of the integral curve
of $\vect{v}$,
\item $\rho$ is a scalar field on $\R^2$,
\item $p\colon\R\to \R$ is a monotone function of $\rho$,
\item $c:=\sqrt{p'(\rho)}$ ($p':=dp/d\rho$)
is called the {\it local speed of sound}.
\item The folowing Euler's equation of motion holds
\begin{equation}\label{eq:euler}
dp+\frac{\rho}2 \, d(|{\mb v}|^2)=0.
\end{equation}
\end{itemize}
We also assume the flow is {\it irrotational}, that is,
\begin{equation}\label{eq:i}
0=\op{rot}(\vect{v})=v_x-u_y,
\end{equation}
where 
$
v_x:=\partial v/\partial x, \,\, u_y:=\partial u/\partial y.
$
Here, \lq the equation of continuity\rq\
is equivalent to the fact that
\begin{equation}\label{eq:c}
0=\op{div}(\rho\vect{v})=(\rho u)_x+(\rho v)_y.
\end{equation}
By \eqref{eq:i}, 
there exists a function $\Phi\colon D\to \R$,
called the {\it potential}
of the flow, such that $\nabla \Phi=\vect{v}$,
where $\nabla \Phi:=(\Phi_x,\Phi_y)$.
Since $p$ is a function of $\rho$,
the fact $c^2=p'(\rho)$ and
\eqref{eq:euler} yield that
\begin{equation}\label{eq:rho-xy}
\rho_x%=\frac{p_x}{c^2}
=-\frac{\rho(uu_x+vv_x)}{c^2},\quad
\rho_y%=\frac{p_y}{c^2}
=-\frac{\rho(uu_y+vv_y)}{c^2}.
\end{equation}
By \eqref{eq:c}, one can easily check that
\begin{equation}\label{eq:c2}
0
%\frac{c^2\bigl((\rho u)_x+(\rho v)_y\bigr)}{\rho}
=(c^2-\Phi_x^2)\Phi_{xx}-2\Phi_x\Phi_y 
\Phi_{xy}+(c^2-\Phi_y^2)\Phi_{yy}.
\end{equation}
On the other hand, by \eqref{eq:c},
there exists a function $\Psi\colon D\to \R$,
called the {\it stream function}
of the flow, such that 
\begin{equation}\label{eq:s}
\Psi_x=-\rho v,\qquad 
\Psi_y=\rho u. 
\end{equation}
If we set $\xi:=\rho u$ and $\eta:=\rho v$,
\eqref{eq:rho-xy} can be written as
$$%\begin{equation}\label{eq:rho-xy2}
(\rho^2c^2-\xi^2-\eta^2)(\rho_x,\rho_y)=-\rho(\xi \xi_x+\eta \eta_x,
\xi \xi_y+\eta \eta_y).
$$%\end{equation}
Since 
$$
0=v_x-u_y=\frac{\eta_x}{\rho}-\frac{\xi_y}{\rho}
-\frac{\eta \rho_x}{\rho^2}+\frac{\xi \rho_y}{\rho^2},
$$
the identity $0=\rho(\xi^2+\eta^2-\rho^2c^2)(v_x-u_y)$
yields that
\begin{align}
\label{eq:i2}
0
=(\rho^2c^2-\Psi_y^2)\Psi_{xx}+
2\Psi_x\Psi_y \Psi_{xy}+(\rho^2c^2-\Psi_x^2)\Psi_{yy}.
\end{align}
A flow satisfying
\begin{equation}\label{eq:C}
\rho c=1
\end{equation}
is called a {\it Chaplygin gas flow}.
For a given stream function
$\Psi\colon D\to \R$ of
the Chaplygin gas flow, we set
\begin{equation}\label{eq:B}
B_\Psi:=1-\Psi_x^2-\Psi_y^2.
\end{equation}

In this paper, for the sake of simplicity,
we abbreviate 
\lq {\bf zero mean curvature}\rq\
by \lq {\bf ZMC}\rq.\
Under the condition \eqref{eq:C},
the equation 
\eqref{eq:i2}
for the stream function $\Psi$ reduces to
\begin{equation}\label{eq:s3}
(1-\Psi_y^2)\Psi_{xx}+2\Psi_x\Psi_y 
\Psi_{xy}+(1-\Psi_x^2)\Psi_{yy}=0.
\end{equation}
We call this the {\it ZMC-equation} in Lorentz-Minkowski 3-space 
$\mb L^3$ of signature $(++-)$.
In fact, at the point where  
$B_\Psi\ne 0$ (cf. \eqref{eq:B})
the mean curvature function $H$ of the graph $t=\Psi(x,y)$
is well-defined, where $(x,y,t)$ are the canonical coordinates
of $\mb L^3$.
Then \eqref{eq:s3} is equivalent to the condition that
$H=0$. 

\begin{definition}
A surface in $\mb L^3$ whose image 
can be locally
expressed as the graph of a certain
$\Psi$ satisfying \eqref{eq:s3} 
after a suitable motion in $\mb L^3$
is called a {\it ZMC-surface}.
A point where $B_\Psi>0$
(resp. $B_\Psi<0$, $B_\Psi=0$) 
is said to be
{\it space-like} (resp. {\it time-like}, {\it light-like}).

A ZMC-surface consisting only of space-like points
is called a {\it maximal surface}.
On the other hand, a surface in $\mb L^3$ consisting only
of light-like points is called a {\it light-like surface}. 
\end{definition}

It is known that the identity $B_\Psi=0$
implies \eqref{eq:s3} (see \cite[Proposition 2.1]{UY2}).
In particular, any light-like surfaces are
ZMC-surfaces in our sense.

If $\rho c=1$, then we have
$
1/{\rho^2}=c^2=dp/d\rho,
$
that is, $dp=d\rho/\rho^2$ is obtained.
Substituting this into \eqref{eq:euler}, we get
$
d(|{\mb v}|^2-{1}/{\rho^2})=0
$
and so there exists a constant $\mu$ such that
\begin{equation}\label{eq:mu-def}
|{\mb v}|^2+\mu=\frac{1}{\rho^2}(=c^2).
\end{equation}
By \eqref{eq:s}, we can rewrite this as
\begin{equation}\label{eq:rho0}
B_\Psi=\mu \rho^2.
\end{equation}
By \eqref{eq:mu-def} and \eqref{eq:rho0},
the sign change of $B_\Psi$ corresponds 
to the type change of the Chaplygin gas flow from 
sub-sonic ($|{\mb v}|<c$) to
super-sonic ($|{\mb v}|>c$),
that is, the sub-sonic part satisfies
$B_\Psi>0$.
If $\mu=0$, then $B_\Psi$ vanishes identically,
and the graph of $\Psi$ gives a light-like surface.
Such surfaces are discussed in the appendix, and
we now consider the case $\mu\ne 0$.
Since $B_\Psi$ and $\mu$ have the same sign (cf. \eqref{eq:rho0}),
we can write
\begin{equation}\label{eq:rho}
\rho=\frac{1}{\sqrt{|\vect{v}|^2+\mu}}
=\sqrt{\frac{1-\Psi_x^2-\Psi_y^2}{\mu}}.
\end{equation}
By \eqref{eq:mu-def} and the fact
$|\vect{v}|^2=\Phi_x^2+\Phi_y^2$, \eqref{eq:c2} can be written as
\begin{equation}\label{eq:c3}
(\mu+\Phi_y^2)\Phi_{xx}-2 \Phi_x \Phi_y 
\Phi_{xy}+(\mu+ \Phi_x^2)\ \Phi_{yy}=0.
\end{equation}
We set
\begin{equation}\label{eq:mu1}
\phi(x,y)
 :=\tilde \mu \Phi(\tilde \mu x,
\tilde \mu y)
\qquad (\tilde \mu:=1/{\sqrt[4]{|\mu|}}).
\end{equation}
If $\mu>0$, then 
\eqref{eq:c3} reduces to
\begin{equation}\label{eq:c3a}
(1+\phi_y^2)\phi_{xx}-2\phi_x\phi_y 
\phi_{xy}+(1+\phi_x^2)\phi_{yy}=0,
\end{equation}
which is known as the condition that
the graph of $\phi(x,y)$ gives a minimal 
surface in the Euclidean 3-space $\mb E^3$.
On the other hand,
if $\mu<0$, then \eqref{eq:c3} reduces to
\begin{equation}\label{eq:c3b}
(1-\phi_y^2)\phi_{xx}+2\phi_x\phi_y 
\phi_{xy}+(1-\phi_x^2)\phi_{yy}=0,
\end{equation}
which is the ZMC-equation (cf. \eqref{eq:s3}).
It can be easily checked that the graph of $\phi$ 
is a time-like ZMC-surface in $\mb L^3$.
In both of the two cases, it can be easily checked that
($\epsilon:=\op{sign}(\mu)\in \{1,-1\}$)
$$
\pmt{\psi_x \\ \psi_y}
=\frac{1}{\sqrt{\phi_x^2+\phi_y^2+\epsilon}}\pmt{-
\phi_y \\ \phi_x}
$$
holds, where
$
\psi:=\Psi(\tilde \mu x,\tilde \mu y)/{\tilde \mu}.
$
Note that $\Psi$ satisfies
\eqref{eq:s3} if and only if 
$\psi$ satisfies
\eqref{eq:s3}.
Moreover, one can easily check that
\begin{equation}\label{eq:hat-rho}
(\hat{\rho}:=)\frac{1}{\sqrt{\phi_x^2+\phi_y^2+\epsilon}}=\sqrt{\epsilon (1-\psi_x^2-\psi_y^2)}.
\end{equation}
This means that
$
\phi \longleftrightarrow \psi
$
corresponds to the duality between potentials
and stream functions of Chaplygin gas flow such that
\begin{itemize}
\item
$\mu=\pm 1(=\epsilon)$, 
\item the density $\hat \rho$ is given as
\eqref{eq:hat-rho}, and 
\item $p=p_0-1/\hat \rho$ for some constant $p_0$.
\end{itemize}
This gives a correspondence 
between graphs of minimal surfaces in $\mb E^3$
and graphs of maximal surfaces in $\mb L^3$ 
(resp.~an involution
on the set of graphs of time-like ZMC-surfaces 
in $\mb L^3$) 
which we call the {\it fluid mechanical duality}.

A part of the above dualities is 
suggested in the classical book \cite{B}.
Calabi \cite{C} also recognized this duality for $\mu>0$, 
and pointed out the following:

\begin{fact}[Calabi's Bernstein-type theorem]\label{fact:C}
Suppose that the graph of a function $\psi\colon \R^2\to \R$ gives
a maximal surface $($that is,
a surface consisting only of space-like points whose 
mean curvature function vanishes identically$)$.
Then $\psi-\psi(0,0)$ is linear.
\end{fact}

This is an analogue of the classical Bernstein theorem 
for minimal surfaces in $\mb E^3$. 
Moreover, Calabi \cite{C} obtained 
the same conclusion for entire space-like ZMC-graphs 
in $\mb L^{n+1}$ ($n\leq 4$), 
and Cheng and Yau \cite{CY} 
extended this result for complete maximal hypersurfaces 
in $\mb L^{n+1}$ for $n\ge 5$. 
The assumption that the graph consists only of space-like points
is crucial. Entire ZMC-graphs which are not planar
actually exist.
Typical such examples are 
of the form
\begin{equation}\label{eq:T-Graph}
\psi_0(x,y):=y+g(x),
\end{equation}
where $g\colon\R\to \R$ is a $C^\infty$-function of one variable.
A point $p=(x_0,y_0)\in \R^2$ is a light-like point of $\psi_0$
if and only if $g'(x_0)=0$.
However, $\psi_0$ does not
contain any space-like points.
The potential function $\phi_0$ corresponding to $\psi_0$ is
given by
$$
\phi_0(x,y)=y-\int_0^x \frac{du}{g'(u)} 
$$
up to a constant.
On the other hand, 
Osamu Kobayashi \cite{K} pointed out 
the existence of entire graphs 
of ZMC-surfaces with space-like points,
light-like points and time-like points all appearing. 
Such a surface is called of {\it mixed type}.
Recently, many such examples
are constructed in \cite{Haw}.

By definition, any entire ZMC-graph 
of mixed type has at least
one light-like point. So we give the
following definition:

\begin{definition}
A light-like point $p$ of the function $\psi$ 
(i.e. $B_\psi(p)=0$)
is said to be
{\it non-degenerate} (resp.~{\it degenerate})
if $\nabla B_\psi$ does not vanish
(resp.~vanishes) at $p$.	
\end{definition}

At each non-degenerate light-like point, 
the graph of $\psi$ changes its causal type
from space-like to time-like.
This case is now well-understood.
In fact, under the assumption that the surface
is real analytic, it can be reconstructed 
from a real analytic null regular curve in $\mb L^3$ 
(cf.\ Gu \cite{G} and also \cite{Okayama,Kl,KKSY}).

On the other hand, 
there are several examples of ZMC-surfaces 
with degenerate light-like points
(cf. \cite{A,A3,CR,HK}).
Moreover, a local general existence theorem for maximal surfaces 
with degenerate light-like points is given in \cite{UY2}.
For  such degenerate light-like points,
we need a new approach to analyze the behavior of 
$\psi$ and $\phi$. The following fact was 
proved by Klyachin \cite{Kl} (see also \cite{UY2}).

\begin{fact}[The line theorem for ZMC-surfaces]\label{fact:IL}
Let $D$ be a domain of $\R^2$ and
$F\colon D\to \mb L^{3}$ a $C^3$-differentiable ZMC-immersion such that 
$o\in D$ is a degenerate light-like point.
Then, there exists a light-like line segment $\hat \sigma\, 
(\subset \mb L^{3})$ 
passing through $F(o)$ of $\mb L^{3}$ such that
$o$ does not coincides with one of  the two end points of  $\hat \sigma$ and 
$F(\Sigma)$ contains $\hat \sigma$, where $\Sigma$ is the set of
degenerate light-like points of $F$.
\end{fact} 

Recently, Fact \ref{fact:IL}  was generalized to a 
much wider class of surfaces, including constant 
mean curvature surfaces in $\mb L^3$, see \cite{UY2,UY3}.
(In \cite{UY2}, the general local existence theorem
of surfaces which changes their causal types
along degenerate light-like lines was also shown.)
The asymptotic behavior of $\psi$ along 
the line $l$ consisting of degenerate light-like points
is discussed in \cite{UY2}.

The purpose of this paper is to prove 
the following assertion:

\medskip
\noindent
{\bf Theorem A.}\label{thm:main}
{\it An entire $C^3$-differentiable ZMC-graph 
which is not a plane
admits a non-degenerate
light-like  point if its space-like part is
non-empty.}
\medskip

This assertion is proved in Section 2
using the fluid mechanical duality
and the half-space theorem for minimal 
surfaces in $\mb E^3$ given by 
Hoffman-Meeks \cite{HM}.
It should be remarked that 
the half-space theorem 
does not hold for time-like ZMC-surfaces.
In fact, the graph of $\phi(x,y):=y+\log{(\tan{x})}$ $(x\in (0,\pi/2))$
gives a properly embedded time-like ZMC-surface
lying between two parallel vertical planes. 
In Section 2, we give further examples, and
provide 
 a few questions related to Theorem A.
As an application, we give the following improvement of
Calabi's Bernstein-type theorem:

\medskip
\noindent
{\bf Corollary B.}
\label{cor:main}
{\it An entire  $C^3$-differentiable ZMC-graph 
which does not admit any time-like points is a plane.}

\medskip
In fact, if the ZMC-graph admits  a
space-like point, then the assertion immediately
follows from Theorem A.
So it remains to show the case that the graph consists
only of light-like points. 
However, such a graph must be a plane,
as shown in the appendix (see
Theorem \ref{cor:A}).

\section{Proof of Theorem A}

In this section,  we prove Theorem A in the introduction.
We let
$
\psi\colon\R^2\to \R
$
be a $C^3$-function satisfying the ZMC-equation \eqref{eq:s3}.
We assume $\psi$ admits a space-like point $q_0\in \R^2$,
but admits no non-degenerate light-like points. 
By Calabi's Bernstein-type theorem (cf. Fact \ref{fact:C}),
$\psi$ has at least one degenerate light-like point.
We set
$$
F_\psi(x,y):=(x,y,\psi(x,y)),
$$ 
which gives the ZMC-graph of $\psi$. 
We denote by $ds^2$ the positive semi-definite metric
which is the pull-back of the canonical Lorentzian metric of $\mb L^3$
by $F_\psi$.
The line theorem (cf. Fact \ref{fact:IL})
yields that the image of $F_\psi$ contains a light-like line segment $\hat \sigma$.
Then the projection of $\hat \sigma$ is a line segment $\sigma$ on the 
$xy$-plane $\R^2$. Then $\sigma$ lies on a line $l$ on $\R^2$.
If $\sigma\ne l$, then there exists an end point $p$
of $\sigma$ on $l$. Since $p$ is the limit point
of degenerate light-like points, $p$ itself is also a
degenerate light-like point.
By applying the line theorem,  there exists a light-like line segment 
$\hat \sigma'$ containing $F_\psi(p)$ as its interior point.
We denote by $\sigma'$ 
the projection  of $\hat \sigma'$ to the $xy$-plane.
Since the null direction at $p$ with respect to the metric $ds^2$
is uniquely determined,
$\sigma'$ also lies on the line $l$. 
Thus, the entire graph contains a whole light-like line
containing $\hat \sigma$.
In particular, degenerate light-like points on the graph
consist of a family of straight lines in $\R^2$.

Let $l$ and $l'$ are two such straight lines.
Then $l'$ never meets $l$.
In fact, if not, 
then there is a unique intersection point
$q\in l\cap l'$.
By Fact \ref{fact:IL}, two lines $l,l'$ can be
lifted to two light-like lines $\tilde l$
and $\tilde l'$ in $\mb L^3$ passing through $F_\psi(q)$.
The tangential
directions of $\tilde l$ and $\tilde l'$ 
are linearly independent light-like vectors
at $F_\psi(q)$. Then by \cite[Lemma 27 in Section 5]{ON},
$q$ is a time-like point, a contradiction.

Thus, the set of degenerate light-like points of $F_\psi$
consists of a family of parallel lines in the $xy$-plane.
Without loss of generality, we may assume that
these lines are vertical and one of them is the $y$-axis. 
Then we can find a domain ($\Delta\in (0,\infty]$)
$$
\Omega:=\{(x,y)\,;\, 0<x<2\Delta\}
$$
such that $q_0\in \Omega$ and
$F_\psi$ has no light-like points on $\Omega$
and both of the lines $l=\{x=0\}$ and $l'=\{x=2\Delta\}$ 
consist of light-like points unless $\Delta=\infty$. 
Since there are no light-like points on $\Omega$,
the potential function $\phi\colon\Omega\to \R$
is induced by $\psi$ as the fluid mechanical dual.
The graph of $\phi$ is a minimal surface in $\mb E^3$.
In particular, $\phi$ is real analytic.
If we succeed to prove that 
the map
$F_\phi(x,y):=(x,y,\phi(x,y))$
is proper, then Theorem A follows.
In fact,  by the half-space theorem given in \cite{HM}
the image $F_\phi(\Omega)$ lies in
a plane in $\mb E^3$. 
Then the map
$
F_\psi(x,y)
$
also lies in a plane $\Pi$ in $\mathbf L^3$ on $\overline{\Omega}$.
Since $F_\psi(l)$ is light-like,
the plane $\Pi$ must be light-like,
contradicting the the fact $q_0\in \Omega$.

To prove the properness of $F_\phi$,
it is sufficient to show the following:

\begin{lemma}
Let $\{p_n\}_{n=1}^\infty$ be a sequence of points
in $\Omega$ accumulating to a point on $l$ or $l'$.
Then $\{|\phi(p_n)|\}_{n=1}^\infty$ diverges.
\end{lemma}

\begin{proof}
By switching the roles of $l$ and $l'$ if necessary,
it is sufficient to consider the case that
$\{p_n\}_{n=1}^\infty$ accumulates to a point on $l$.
Taking a subsequence and using a suitable
translation of the $xy$-plane, we may
assume that $\{p_n\}_{n=1}^\infty$ 
converges to the origin $(0,0)\in l$
and  $p_n=(x_n,y_n)$ $(n=1,2,3,...)$ satisfy
the following properties:
\begin{itemize}
\item there exists $\epsilon>0$ such that
$|y_n|<\epsilon$ for each $n=1,2,\dots$
and 
\item there exists $(\delta,0)\in \Omega$ 
($\delta>0$) such that
$$
\delta>x_1>x_2 >\cdots >x_n>x_{n+1}> \cdots.
$$
\end{itemize}
Since $l$ consists of
degenerate light-like points, 
there exists a neighborhood $U$ of $(0,0)$  such that
(see \cite{CR} or \cite[(6.1)]{UY2})
\begin{equation*}
\psi(x,y)=y + x^2 h(x,y) \qquad ((x,y)\in U),
\end{equation*}
where $h(x,y)$ is a  $C^1$-differentiable
function defined on $U$ (see \cite[Appendix A]{UY2}).
Taking $\epsilon,\delta$ to be sufficiently small,
we may assume that
$$
V:=\{(x,y)\in \Omega\,;\, |x|\le \delta,\,\, |y|<\epsilon\}\subset U.
$$ 
Since $B_\psi>0$,
the potential function $\phi$
associated to $\psi$ satisfies
(cf. \eqref{eq:hat-rho})
$$
\phi_x=\frac{\psi_y}{\rho}, 
\qquad \rho=\sqrt{1-\psi_x^2-\psi_y^2}.
$$
Since
$$
1-\psi_x^2-\psi_y^2
=-x^2\biggl((2h+xh_x)^2+2h_y+ x^2 h_y^2\biggr)
$$
is non-negative on the closure $\overline{V}$
of $V$, we can write
\begin{equation}\label{eq:def-k}
\sqrt{\rho}=|x| k(x,y), 
\end{equation}
where $k(x,y)$ is a 
non-negative continuous 
function defined on $\overline{V}$
such that $k$ is positive-valued
on $V$.
We set
$
p_0:=(\delta,0),
$
and consider the path
$\gamma_n\colon[0,1]\to V$
defined by
$\gamma_n(s):=(\delta,2s y_n)$
if $0\le s\le 1/2$ and 
$$
\gamma_n(s):=(2(x_n-\delta)s-x_n+2\delta,y_n) 
$$
if $1/2\le s\le 1$,  
which starts at $p_0$ and terminates at $p_n$.
This curve $\gamma_n$ is the union of the vertical subarc $\gamma_{n,1}$ and
the horizontal subarc $\gamma_{n,2}$. 
So we can write 
\begin{align*}
\phi(p_n)-\phi(p_0)&=\int_{\gamma_n}\phi_x dx +\phi_y dy\\
&=\int_{\gamma_{n,2}}\phi_x dx +\int_{\gamma_{n,1}} \phi_y dy.
\end{align*}
Since $[-\epsilon,\epsilon]\ni y \mapsto \phi_y(\delta,y)\in \R$
is a continuous function, we have that
\begin{align*}
\left|\int_{\gamma_{n,1}} \phi_y dy\right|
&\le  \int_{\gamma_{n,1}} \biggl|\phi_y(\delta,2t y_n)\biggr|\, |dy| \\
&\le \epsilon \max_{|y|\le \epsilon}\,\biggl|\phi_y(\delta,y)\biggr|<\infty.
\end{align*}
So to prove the lemma, it
is sufficient to show that
$\int_{\gamma_{n,2}} \phi_x dx$ diverges as $n\to \infty$.
We set
$$
m:=\max_{x\in [0,\delta],\,\, |y|\le \epsilon}k(x,y)\,\,(\ge 0),
$$
where $k$ is the continuous function given in \eqref{eq:def-k}.
On the other hand, we can  take a 
constant $m'(>0)$
such that
$$
\psi_y=1+x^2h_y(x,y)>m'
\qquad (x\in [0,\delta],\,\, |y|\le \epsilon),
$$
since  $\epsilon,\delta$ can be chosen to be sufficiently small. 
Since $\phi_x=\psi_y/\rho$, we have
\begin{align*}
\left|\int_{\gamma_{n,2}}\phi_x dx \right| 
&=
\int_{x_n}^{\delta} \frac{1+x^2h_y(x,y)}{x^2k^2(x,y)} dx \\
&>
\frac{m'}{m^2}\int_{x_n}^{\delta} \frac{dx}{x^2}
=
\frac{m'}{m^2}\left(\frac{1}{x_n}-\frac{1}{\delta}\right)\to \infty
\end{align*}
proving the assertion.
\end{proof}

\begin{remark}
In the above proof, we showed that $F_\psi(\Omega)$ 
lies in a plane using the fluid mechanical duality.
We remark here that this
can be proved by a different method.
In fact, $\psi$ satisfies the assumption of 
Ecker \cite[Theorem G]{E} or is an 
PS-graph on the convex domain $\Omega$ in the
sense of Fernandez and Lopez \cite{FL}. 
Thus, we can conclude that
$\psi(\Omega)$ lies in a light-like plane.
\end{remark}

\begin{figure}[htb]
 \begin{center}
   \begin{tabular}{c@{\hspace{0.5cm}}c}
        \includegraphics[height=2.5cm]{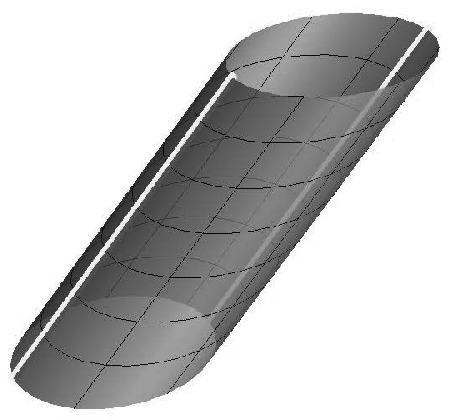} &
        \includegraphics[height=2.5cm]{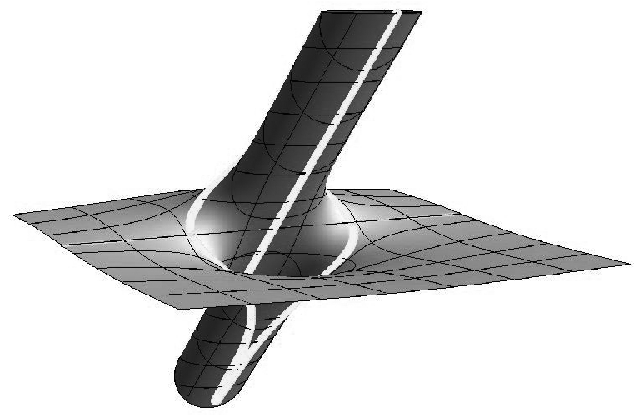} 
    \end{tabular}
\caption{The ZMC-surfaces in Example \ref{ex:1} (left) 
and in Example \ref{ex:2} (right), where 
the white lines 
indicate light-like points.}
\label{fig:0}
 \end{center}
\end{figure}

In \cite{A}, the first author constructed several
ZMC-surfaces  foliated by circles and at most countably many straight lines.
At the end of this paper, we pick up two
important examples of them which contain degenerate 
light-like points. (In \cite{A}, these two examples 
are not precisely indicated. Here we show their 
explicit parametrization and embeddedness.)

\begin{example}[{\cite[Figure 5]{A}}]\label{ex:1}
We set
$$
F(u,v):=(u+a \cos v, a \sin v,u),
$$
where $a>0$ and 
$(u,v)\in \R\times [0,2\pi)$.
Then the image of $F$ contains two parallel 
degenerate light-like lines which correspond to
the special values $\theta=\pm \pi/2$ 
(see Figure \ref{fig:0}, left).
The image of $F$ can be characterized by the
implicit function
$
(x-t)^2+y^2=a^2.
$
This ZMC-surface 
is properly embedded and is not simply connected.
\end{example}

\begin{example}[{\cite[Figure 2]{A}}]\label{ex:2}
We set
$$
F(r,\theta):=\biggl(r+
\frac1{2a}\log\left(\frac{ar-1}{ar+1}\right)+r \cos \theta, 
r\sin \theta, \frac1{2a}\log\left(\frac{ar-1}{ar+1}\right)
\biggr),
$$
where $a>0$ and 
$\theta\in [0,2\pi)$.
This map is defined for $r>1/a$ or $r<-1/a$, 
and the closure of the image of
$F=(x,y,t)$ can be expressed as
$$
(\Psi:=)a \sinh (a t) 
\left((x-t)^2+y^2\right)
+2 (x-t) \cosh (a t)=0.
$$
It can be checked that 
$(\Psi_x,\Psi_y,\Psi_t)$ never vanishes along $\Psi=0$.
So the closure of $F$ gives a
properly embedded ZMC-surface in $\mb L^3$
(see Figure \ref{fig:0}, right).
\end{example}

Regarding our main result, we state a few open problems:

\medskip
\noindent
{\bf (Question 1.)}
{\it Does a properly embedded ZMC-surface which 
consists only of space-like or light-like points
coincide with a plane?
}

\medskip
If this question is affirmative, then Corollary B
follows as a corollary.
Suppose that we can find such a non-planar
ZMC-surface $S$, it must contain a light-like line. In fact, if 
$S$ consists only of space-like points,
then $S$ is complete, and such a surface
must be a plane (see \cite[Remark 1.2]{UY1}).
So $S$ has a light-like point $p$.
If $p$ is non-degenerate, then $S$ has
a time-like point near $p$, so
$p$ must be degenerate. 
By the line theorem (Fact \ref{fact:IL}),
$S$ must contain a
light-like line consisting of degenerate
light-like points.

\medskip
\noindent
{\bf (Question 2.)}
{\it Are there entire ZMC-graphs of mixed type
containing degenerate light-like points?}

\medskip
This question needs to consider 
ZMC-graphs of mixed type. In fact,
if we choose a function $g(x)$ satisfying $g'(0)=0$
as in \eqref{eq:T-Graph}, 
then the $y$-axis consists of the degenerate 
light-like  points.
If we weaken \lq entire ZMC-graphs\rq\ to
\lq properly embedded ZMC-surfaces of mixed type\rq\,
the answer is \lq yes\rq. 
In fact, Example \ref{ex:2}
gives a properly embedded ZMC-surface of mixed type 
which contains a degenerate light-like line $L$.
Although  the space-like points 
never accumulate to $L$ in the case of this example, 
one can show 
the existence of a function $\psi :U\to \R$
defined on a domain $U$ in $\R^2$
containing the $y$-axis 
such that 
\begin{itemize}
\item the $y$-axis corresponds to  a degenerate light-like 
line,
\item $\psi$ is of mixed type, 
or consisting only of space-like points except along the $y$-axis.  
\end{itemize}
See \cite{AUY} for details.
Also, the following question arises: 

\medskip
\noindent
{\bf (Question 3.)}
{\it Are there entire ZMC-graphs of mixed type
which are not obtained as analytic extensions of 
Kobayashi surfaces given as in \cite{Haw}?
}

\medskip
In fact, all  known examples of entire 
ZMC-graphs of mixed type are 
obtained as analytic extensions of 
Kobayashi surfaces (cf.~\cite{Haw}),
and they admit only non-degenerate light-like points. 

\medskip
\appendix
\section{A property of light-like surfaces in $\mb L^3$}

It can be easily checked that
an embedded surface $S(\subset \mb L^3)$ is light-like
if and only if 
the restriction of the canonical
Lorentzian metric on $\mb L^3$
to the tangent space $T_pS$ of
each $p\in S$ is positive semi-define but not
positive definite. 
The purpose of this appendix is to prove the following:

\begin{theorem}\label{cor:A}
If an entire $C^2$-differentiable 
graph of $\psi\colon\R^2\to \R$
gives a light-like surface in $\mb L^3$,
then $\psi-\psi(0,0)$ is a linear function.
\end{theorem}

We set
$
F(x,y)=(x,y,\psi(x,y)).
$
Since $F$ is a light-like surface,
$\psi_x^2+\psi_y^2=1$ holds on $\R^2$.
Differentiating this with respect to $x$
and $y$, we get two equations. 
Since $F$ is light-like, $(\psi_x,\psi_y)\ne (0,0)$.
By thinking $\psi_x,\,\, \psi_y$ are unknown variables
of these two equations,
the determinant $\psi_{xx}\psi_{yy}-\psi_{xy}^2$
vanishes identically.
So the Gaussian curvature of $F$ 
with respect to the Euclidean metric of $\R^3$ 
vanishes identically.
Then, by the Hartman-Nirenberg cylinder theorem,
$F$ must be a cylinder.  
(The proof of the cylinder theorem in \cite{HN} needs only
$C^2$-differentiability). 
That is, there exist a non-zero vector $\mb a$,
a plane $\Pi$ which is not parallel to  $\mb a$
and a regular curve 
$\gamma\colon \R\to \Pi$
such that
$
F(u,v):=\gamma(u)+v \mb a
$
gives a new parametrization of $F$.
If $F$ is not a plane, there exists
$u_0\in \R$ such that
$\gamma'(u_0)$ and $\gamma''(u_0)$
are linearly independent.
Then the point $(u,v)=(u_0,0)$ is not an umbilical 
point of $F$. Since the asymptotic direction 
is uniquely determined at each non-umbilical 
point on a flat surface, 
the line theorem (cf. Fact \ref{fact:IL}) yields 
that $\mb a$ is a light-like vector.
By a suitable homothetic transformation
and an isometric motion in $\mb L^3$, 
we may set $\mb a :=(1,0,1)$.
Then it holds that
\begin{equation}\label{eq:xt}
0=\gamma'\cdot \mb a=x'-t'.
\end{equation}
Since $\gamma'\cdot \gamma'=0$, we have
$y'=0$. So, with out loss of generality, we may assume $y(u)=0$.
Differentiating \eqref{eq:xt}, we have
$x''-t''=0$, contradicting the  fact that 
$\gamma'(u_0)$ and $\gamma''(u_0)$
are linearly independent.
Thus $F$ is a plane.
\qed

\begin{acknowledgement}
The authors would like to express their gratitude to
Atsufumi Honda for fruitful discussions.
\end{acknowledgement}


\begin{thebibliography}{00}

\bibitem{A}
 S. Akamine, {\it
Causal characters of zero mean curvature surfaces of Riemann 
type in Lorentz-Minkowski 3-space},
Kyushu J. Math., {\bf 71} (2017), 211-249. 

\bibitem{A3}
S. Akamine and R.K. Singh, {\it Wick rotations of solutions to the minimal 
surface equation, the zero mean curvature equation and the Born-Infeld 
equation}, to appear in Proc. Indian Acad. Sci. Math. Sci.

\bibitem{AUY}
S. Akamine, M. Umehara and K. Yamada,
{\it Space-like zero mean curvature surface containing
a complete degenerate light-like line in the Lorentz-Minkowski 3-space}, preprint.

\bibitem{B}
L. Bers, 
Mathematical aspects of subsonic and transonic gas dynamics, 
John Wiley \& Sons, 1958.

\bibitem{C}
E. Calabi, {\it Examples of Bernstein problems 
for some nonlinear equations 
in Global Analysis}, 
(Proc. Sympos. Pure Math., Vol. XV, Berkeley, CA, 1968), Amer. Math. Soc.,
Providence, RI, 1970, 223--230.

\bibitem{CY}
S.~Y.~Cheng  and S.~T.~Yau,  {\it Maximal space-like hypersurfaces 
in the Lorentz-Minkowski spaces}, Ann. Math. {\bf 104} (1976) 407--419.


\bibitem{E}
K.~Ecker,
{\it Area minimizing hypersurfaces in Minkowski space},
Manuscripta Math. {\bf 56} (1986), 375--397. 

\bibitem{FL}
I.~Fernandez and  F.~J.~ Lopez,
{\it
On the uniqueness of the helicoid and  Enneper's
surface in the Lorentz-Minkowski space $\R^3_1$,
Trans. Amer. Math. Soc.
363 (2011), 
4603--4650.}


\bibitem{Haw}
S.~Fujimori, Y.~Kawakami, M.~Kokubu, W.~Rossman, M.~Umehara, K.~Yamada,
{\it 
Entire zero mean curvature graphs of
mixed type in Lorentz-Minkowski 3-space},
The Quarterly J. Math. {\bf 67} (2016),
801--837.


\bibitem{CR}
  S. Fujimori,
  Y.W.~Kim,
  S.-E.~Koh,
  W.~Rossman,
  H.~Shin, H.~Takahashi,
  M.~Umehara,
  K.~Yamada and
  S.-D.~Yang,
  {\it 
 Zero mean curvature  surfaces in $\mb L^3$ containing 
 a light-like line},  
C.R. Acad. Sci. Paris. Ser. I. {\bf 350} (2012), 975--978.


\bibitem{Okayama}
S. Fujimori, Y. W. Kim, S.-E. Koh, W. Rossman, H. Shin, M. Umehara,
K. Yamada and S.-D. Yang,
{\it Zero mean curvature surfaces 
       in Lorentz-Minkowski $3$-space and $2$-dimensional
       fluid mechanics}, 
Math. J. Okayama Univ. {\bf 57} (2015), 173--200.


\bibitem{G}
C.H. Gu, 
{\it The extremal surfaces in the 3-dimensional 
Minkowski space}, Acta Math. Sinica (N.S.) {\bf 1} (1985), 173--180.

\bibitem{HN}
P.~Hartman, and L. Nirenberg, 
{\it On spherical image whose Jacobians do not change sign},
Amer. J. Math. {\bf 81} (1959), 901--920. 

\bibitem{HK}
K.~Hashimoto and S.~Kato,
{\it Bicomplex extensions of zero mean curvature surfaces
in $\R^{2,1}$ and $\R^{2,2}$}, 
J.~Geom.~Phys.
{\bf 138}  (2019), 223--240.


\bibitem{HM}
D. Hoffman and W. Meeks, 
{\it The strong half-space theorem for minimal surfaces}, 
Invent. Math. {\bf 101} (1990), 373--377. 


\bibitem{KKSY}
 Y. W. Kim, S.-E Koh, H. Shin and S.-D. Yang,
  {\itshape Spacelike maximal surfaces, 
timelike minimal surfaces, 
and Bj\"{o}rling representation formulae}, 
  J. Korean Math. Soc. {\bf 48} (2011), 1083--1100. 

\bibitem{Kl}
 V.~A.~Klyachin, 
 {\itshape 
Zero mean curvature surfaces of 
mixed type in Minkowski space}, 
	Izv. Math.  {\bf 67} (2003), 209--224.


\bibitem{K}
 O.~Kobayashi, 
  {\it Maximal surfaces in the $3$-dimensional Minkowski space 
    $L^3$}, Tokyo J. Math., {\bf 6} (1983), 297--309.


\bibitem{ON}
B.~O.~Neill,
{\it Semi-Riemannian Geometry}, Academic Press 1983, USA.


\bibitem{UY1}
M. Umehara and K. Yamada, {\it Maximal surfaces with singularities 
in Minkowski space},
Hokkaido Math. J. {\bf 35} (2006), 13--40.

\bibitem{UY2}
 M. Umehara and K. Yamada, {\it 
Surfaces with light-like points
in Lorentz-Minkowski space with
applications}, 
in \lq\lq Lorentzian Geometry and Related Topics",
Springer Proceedings
of Mathematics \& Statics {\bf 211}, 253-273, 2017.

\bibitem{UY3}
 M. Umehara and K. Yamada, {\it 
Hypersurfaces with light-like points 
in a Lorentzian manifold},
to appear in J.~Geom.~Anal.  (arXiv:1806.09233) 

\end{thebibliography}
\end{document}